\documentclass[a4paper,12pt]{article}

\usepackage[margin=1.05in]{geometry}
\usepackage[utf8]{inputenc}

\usepackage[T1]{fontenc}

\usepackage{xfrac}
\usepackage{braket}
\usepackage{enumitem}

\usepackage{amsmath}
\usepackage{amsthm}
\usepackage{amstext}
\usepackage{amssymb}
\usepackage[colorlinks, citecolor=blue]{hyperref}
\usepackage{cleveref}
\usepackage{tikz}
\usetikzlibrary{matrix,arrows,mindmap,backgrounds}

\usepackage{mathtools}
\usepackage{stmaryrd}
\usepackage{mathrsfs}
\usepackage{bbold}
\usepackage{extarrows}
\usepackage{graphicx, subfigure}
\usepackage{fancyhdr}
\usepackage{marginnote}
\usepackage{xifthen}
\usepackage{enumitem}
\usepackage{verbatim}
\usepackage{changepage}
\usepackage{csquotes}

\usepackage{marginnote}
\usepackage[nolist,nohyperlinks]{acronym}

\newtheoremstyle{mythm}
  {\topsep} 
  {\topsep} 
  {\itshape} 
  {} 
  {\bfseries} 
  {.} 
  {.5em} 
  {} 

\theoremstyle{plain}
\newtheorem{thm}{Theorem}
\crefname{thm}{theorem}{theorems}
\newtheorem{prop}[thm]{Proposition}
\crefname{prop}{proposition}{propositions}
\newtheorem{lemma}[thm]{Lemma}
\newtheorem{cor}[thm]{Corollary}

\theoremstyle{definition}
\newtheorem{definition}[thm]{Definition}
\newtheorem{remark}[thm]{Remark}
\newtheorem{corollary}[thm]{Corollary}

\crefname{construction}{Construction}{Constructions}

\newcommand{\norm}[1]{\left\vert#1\right\vert}
\newcommand{\normg}[1]{\left\lVert#1\right\rVert_g}

\newcommand{\RR}{\mathbb{R}}
\newcommand{\NN}{\mathbb{N}}
\newcommand{\CC}{\mathbb{C}}
\newcommand{\PP}{\mathbb{P}}
\newcommand{\QQ}{\mathbb{Q}}
\newcommand{\CP}{\CC\PP}
\newcommand{\TT}{\mathbb{T}}
\newcommand{\ZZ}{\mathbb{Z}}
\newcommand{\SSS}{\mathbb{S}}
\newcommand{\calF}{\mathcal{F}}

\newcommand{\F}{{\mathcal{F}}}
\newcommand{\G}{{\mathcal{G}}}
\newcommand{\R}{{\mathbb{R}}}
\newcommand{\N}{{\mathbb{N}}}
\renewcommand{\d}{{\operatorname{d}}}

\newcommand{\Op}{{\mathcal{O}p}}
\newcommand{\wtd}{\widetilde}

\newcommand{\depth}{{\operatorname{depth}}}

\newcommand{\calN}{\mathcal{N}}
\newcommand{\calK}{\mathcal{K}}
\newcommand{\sint}{{\operatorname{int}}}
\newcommand{\inj}{{\operatorname{inj}}}

\title{Examples of symplectic non-leaves}
\author{Fabio Gironella\footnote{Humboldt University, Berlin, Germany. Email: \url{fabio.gironella.math@gmail.com}} \and Lauran Toussaint\footnote{Université Libre de Bruxelles, Brussels, Belgium. Email: \url{lauran.toussaint@ulb.be}}}
\date{}

\begin{document}

\maketitle

\begin{acronym}
    \acro{sgb}[SGB]{strongly geometrically bounded}
    \acro{gb}[GB]{geometrically bounded}
\end{acronym}

\begin{abstract}
This paper deals with the following question: which manifolds can be realized as leaves of codimension-$1$ symplectic foliations on closed manifolds?
We first observe that leaves of symplectic foliations are necessarily strongly geometrically bounded. We show that a symplectic structure which admits an exhaustion by compacts with (convex) contact boundary can be deformed to a strongly geometrically bounded one.
We then give examples of smooth manifolds which admit a strongly geometrically bounded symplectic form and can be realized as a smooth leaf, but not as a symplectic leaf for any choice of symplectic form on them.
Lastly, we show that the (complex) blowup of $2n$-dimensional Euclidean space at infinitely many points, both admits strongly geometrically bounded symplectic forms for which it can and cannot be realized as a symplectic leaf.
\end{abstract}

\section*{Introduction}
\label{sec:intro}

The realizability problem, namely that of understanding which (open) manifolds can be leaves of foliations\footnote{Unless explicitly stated otherwise, throughout this paper ``foliation'' will always mean smooth codimension-$1$ foliation on a closed ambient manifold. Similarly, the realizability problem is intended for open manifolds.}, was introduced by Sondow \cite{Son75} and has been extensively studied since then. 

The situation in low dimensions is very flexible: every open surface is a leaf of a foliation on any given ambient manifold \cite{CanCon87}.
In higher dimensions, there are many manifolds which are not diffeomorphic to leaves of foliations. The first examples were found by \cite{Ghy85,INTT85}, subsequently simply connected examples were given in \cite{AttHur96,SchSou17}. More recently, \cite{MCS18b,MCS20} have found some examples of topological manifolds with exotic smooth structures which are not diffeomorphic to smooth leaves.
One can also consider foliations with additional leafwise structures; see for instance \cite{Zeg94,AttHur96,Sch11} for examples of Riemannian manifolds not quasi-isometric to leaves of a Riemannian foliation.
In the symplectic setting Bertelson \cite{Ber01} found foliations which do not admit a leafwise symplectic structure (although the formal obstructions vanish). 

Here we investigate obstructions on the level of a single leaf. That is we consider the following question:
\[
\parbox{0.9\textwidth}{\emph{Is there a (symplectic) manifold that is not a leaf of a symplectic foliation?}}
\]

This question makes sense both with and without fixing the symplectic structure on the manifold. 
Furthermore, any closed symplectic manifold $(W,\omega)$ can be realized as a symplectic leaf by considering the product foliation on $W \times \SSS^1$.  
As such we restrict our attention to \emph{open} manifolds.
Even though $W$ is open, the compactness of the ambient manifold implies that the behavior of $\omega$ resembles the compact case:
any leaf $(W,\omega)$ is \emph{strongly geometrically bounded} (or \acs{sgb} in short) if there exists a compatible almost complex structure $J$ for which the associated Riemannian metric $g= \omega(\cdot, J\cdot)$ has bounded injectivity radius and scalar curvature (see Definition \ref{def:str_geom_bounded} and \Cref{lemma:sympl_leaves_sgb} for details).\footnote{As the name suggests \acs{sgb} is stronger, and implies, the usual notion of geometrically boundedness (see e.g.\ \cite[Definition 4.1.1]{Sik91}).}
We point out that non-\acs{sgb} (and non-geometrically-bounded) symplectic manifolds are plentiful: for instance the open ball $(D^{2n},\omega_{std})$ or any symplectizations of contact manifolds  
are not \acs{sgb} and thus cannot be a leaf of a symplectic foliation.

In light of this we consider the following variations of our motivating question (in which $W$ is always assumed to be open):

\begin{enumerate}[label=Q\arabic*.]
    \item\label{Q2} Is there a manifold that admits a \acs{sgb} symplectic structure but is not diffeomorphic to a symplectic leaf (for any choice of $\omega$)? 
    \item\label{Q3} Is there a \acs{sgb} symplectic manifold $(W,\omega)$ with $W$ diffeomorphic to a symplectic leaf but $(W,\omega)$ not symplectomorphic to a symplectic leaf?
\end{enumerate}

In this paper we answer positively to both questions.
We start by proving a symplectic analogue of Greene's result \cite{Gre78} on the existence of Riemannian metrics of bounded geometry on open manifolds. 
In order to state our result, we define \emph{exhaustion of contact type} any exhaustion by compacts $\mathcal{K} = \{K_n\}_{n \in \NN}$ on a symplectic manifold $(W,\omega)$ such that $\partial K_n$ is a hypersurface of (convex) contact type.
We then prove the following:

\begin{thm}
    \label{thm:strongly_geom_bounded}
    Let $(W,\omega)$ be a symplectic manifold admitting an exhaustion $\mathcal{K}$ of contact type.
    Then  $\omega$ is homotopic through symplectic forms to a \acs{sgb} symplectic form $\omega'$.
\end{thm}
We point out that, on each of the compact sets $K_n \setminus \Op(K_{n-1})$, the symplectic form $\omega'$ (and the whole homotopy) is just a rescaling of $\omega$ by a constant depending on $n$; in particular, the less obvious part of the construction happens near the boundaries of the $K_n$'s, where $\omega$ simply gets ``stretched'' along a collar (see the proof in \Cref{sec:geom_bound} for more details).
Notice also that the above result is well-known for Liouville manifolds of finite-type, see \cite[Proposition 2.2 and Remark 2.3]{CGK04}; more precisely, in the finite-type case the homotopy is not necessary, and $\omega$ is directly \acs{sgb}.

In order to answer Question \ref{Q2}, we prove:
\begin{thm}\label{thm:smooth_manifold_sympl_nonleaf}
    Let $W^{2n}$ be an open manifold with a finite number $k$ of ends. 
    Suppose that, for $i=1,\dots,k$, the $i$-th end has a neighborhood of the form $N_i \times [0,\infty)$, where $N_i$ has trivial $\pi_1$  and $H^2$. 
    If there is only one end, \emph{additionally} assume that the compact $W \setminus N_1 \times (0,\infty)$ has non-trivial $\pi_1$ or non-trivial $H^2$. 
    Then $W$ is not diffeomorphic to a leaf of a symplectic foliation.
\end{thm}

Notice that these manifolds have cylindrical ends and as such they can be realized as leaves of a smooth foliation. For example, consider the truncation $W_{tr}$ of $W$, which is the compact manifold with boundary obtained by removing the cylindrical end. 
Using a turbulization procedure (see for example \cite[Example 3.3.11]{CanCon00}), we obtain a foliation on $\SSS^1 \times W_{tr}$ (tangent to the boundary) whose non-compact leaves are all diffeomorphic to $W$. Hence, gluing two such Reeb components gives a foliation on a closed manifold containing $W$ as a leaf.\\
Moreover, it is not hard to find examples of such manifolds admitting a \acs{sgb} symplectic structure, thus giving a positive answer to Question \ref{Q2}:

\begin{cor}
\label{cor:smooth_weinstein_nonleaf}
Let $(W,\omega)$ be a Weinstein manifold of finite type without $1$-and $2$-handles, and denote by $(W',\omega')$ the symplectic blowup at any point of $W$. Then, $\omega'$ is \acs{sgb} but $W'$ is not diffeomorphic to a symplectic leaf.
\end{cor}
In particular, for instance $\R^{2n}$ blown-up at a point admits a $\acs{sgb}$ symplectic structure but is not diffeomorphic to a symplectic leaf. 
In fact, although $(\R^{2n},\omega_{st})$ itself does not satisfy the conditions of Theorem \ref{thm:smooth_manifold_sympl_nonleaf}, part of the proof of the latter (and an explicit construction) gives the following:
\begin{cor}
\label{cor:R2n_non_proper_leaf}
$\RR^{2n}$ is not diffeomorphic to a proper symplectic leaf. 
On the other hand, $(\RR^{2n},\omega_{std})$ can be realized as a non-proper symplectic leaf.
\end{cor}

Recall that a leaf $L$ is \emph{proper} if it does not accumulate onto itself.

Let's now go back to \ref{Q3}, to which we give a positive answer using again the symplectic blow-up procedure, but this time at infinitely many points. 

\begin{thm}
    \label{thm:blowup_non_leaves}
    Let $(W,\omega)$ be a symplectic manifold admitting an exhaustion of contact type.
    Then there is a \acs{sgb} symplectic manifold $(W',\omega')$, with $W'$ obtained from $W$ by (complex) blowup at infinitely many points, which is not symplectomorphic to a leaf of a symplectic foliation.
\end{thm}
More precisely, if $\{K_n\}_{n \in \N}$ is the given exhaustion of contact type, then we symplectically blowup at a countable sequence of points $\{p_n\}_{n \in \N}$ where $p_n \in K_{n} \setminus K_{n-1}$. 
At each of these points the blowup is performed so that the resulting copies of $\CP^{n-1}$ have different $\omega'$-volumes. 
The difference in volume is essential as the result of blowing up with the same size can sometimes be realized as a leaf. 
Indeed, we have the following positive answer to Question \ref{Q3}:

\begin{cor}
\label{cor:real_space_blown_up}
Let $W$ be the result of (complex) blowing up $\R^{2n}$ at a sequence of points going to infinity. 
Then $W$ is diffeomorphic to a symplectic leaf, but also admits a \acs{sgb} symplectic form $\omega$ for which $(W,\omega)$ is not symplectomorphic to a symplectic leaf.
\end{cor}

The fact that $W$ is diffeomorphic to a symplectic leaf follows from an explicit construction. 
In this case, the induced volumes of the $\CP^{n-1}$ resulting from blowing up are all the same. 
In the second part of the conclusion, $\omega$ is constructed so that the induced volumes are on the contrary different, so that the statement follows from Theorem \ref{thm:blowup_non_leaves}.

We emphasize again that unless explicitly stated otherwise, throughout this paper ``foliation'' will always mean smooth codimension-$1$ foliation on a closed ambient manifold.

\paragraph*{Outline} In \Cref{sec:ends_and_accumulation} we recall some notions and results from smooth foliation theory, describe the symplectic (almost-)periodicity notion for ends of leaves of symplectic foliations, and prove that proper leaves of symplectic foliations are ``symplectically almost periodic'' (see \Cref{thm:sympl_proper_leaves_approx_periodic_end}).
In \Cref{sec:geom_bound} we define strongly geometrically bounded symplectic manifolds, and prove \Cref{thm:strongly_geom_bounded}.
In \Cref{sec:not_diffeo_sympl_leaves} we describe topological obstructions for a manifold to be diffeomorphic to a proper or non-proper symplectic leaf, proving in particular \Cref{thm:smooth_manifold_sympl_nonleaf} and \Cref{cor:R2n_non_proper_leaf} above.
Lastly, in \Cref{sec:nonleaves_blowup} we study examples coming from the symplectic blowup construction, thus proving \Cref{thm:blowup_non_leaves} and \Cref{cor:real_space_blown_up}.

\subsection*{Acknowledgements}
We are grateful to Mélanie Bertelson for pointing out to us the question motivating this paper, i.e.\ of what symplectic manifolds can arise as leaves in symplectic foliations on closed manifolds, as well as for very useful discussions in relation to the examples in \Cref{thm:smooth_manifold_sympl_nonleaf}. 
We would also like to thank Klaus Niederkr\"uger for communicating to us the observation, due to Fran Presas, that leaves of symplectic foliations on closed manifolds are geometrically bounded, and for sending to us a draft of his work \cite{AlbNieInPrep}, which elaborates on the consequences of this observation by laying the foundations for the use of pseudo-holomorphic curves on (strong) symplectic foliations relying on this observation.
This work is partially supported by the ERC Grant ``Transholomorphic''.

\section{Ends of manifolds and accumulation of leaves}
\label{sec:ends_and_accumulation}

Intuitively, the ends of a manifold represent the (topologically) different ways to go to infinity.
In order to make this precise consider an exhaustion by compacts $\calK=\{K_i\}_{i \in \N}$ of a manifold $W$. The \emph{endset} $\mathcal{E}_\calK(W)$ is the set of sequences
\[ U_1 \supset U_2 \supset U_3 \supset \dots,\]
where $U_i$ is a connected component of $W \setminus K_i$;
each element $e$ of $\mathcal{E}_\calK(W)$ is called an \emph{end} of $W$.
One can see that the endsets $\mathcal{E}(W)$ associated to two different exhaustion by compacts are in natural bijection; therefore, we will just denote it $\mathcal{E}(W)$ from now on.
We lastly call \emph{neighborhood} of an end $e = \{U_i\}_{i \in \N}$ any open set $V$ such that $U_n \subset V$ for some $n$.

A leaf $L$ of a foliation, or more precisely an end $e$ of $L$, is said to \emph{accumulate} onto a leaf $L'$ if
\[ \overline{U} \cap L' \neq \emptyset,\]
for some (and hence any) neighborhood $U$ of $e$. Since on a closed manifold any non-compact leaf must accumulate we obtain the following invariant. A leaf $L$ is said to be at \emph{depth $0$} if it is compact, and at \emph{depth $k$} if $\overline{L} \setminus L$ is a union of leaves at depth $<k$; lastly, $L$ is at \emph{infinite depth} if it is not at any finite depth.

The way leaves at finite depth can accumulate closely resembles the way the interior leaves of a Reeb component spiral around the boundary. 
That is, each end has a neighborhood which spirals (Definition \ref{def:Spiraling}) onto a leaf at lower level.

\begin{thm}[{\cite[Theorem 8.4.6]{CanCon00}}]
\label{thm:SpiralingTheorem}
If $L$ is a leaf at depth $k$, then
\[ L = A \cup B^1\cup\dots \cup B^q,\]
where $A$ is a compact, connected, $(n-1)$-dimensional manifold with boundary components $N^1,\dots,N^q$ and
\begin{enumerate}
\item $A \cap B^j = N^j$, $1 \leq j \leq 1$;
\item $B^i \cap B^j = \emptyset$, $i \neq j$;
\item $B^j$ spirals on a leaf $L^j$ at depth \footnote{Note that the statement in \cite{CanCon00} talks about levels instead of depth. However, in light of \cite[Corollary 8.3.16]{CanCon00} they are equivalent.} at most $k-1$, $1 \leq j \leq q$;
\item for at least one value of $j$, $L^j$ is at depth $k-1$.
\end{enumerate}
\end{thm}

Before stating the precise definition of spiraling let us point out two properties of the spiraling in a Reeb component. 
Firstly, on a neighborhood of the boundary there is a projection (e.g.\ along the leaves of an auxiliary transverse $1-$dimensional foliation) onto the boundary leaf. 
Secondly, the end of each interior leaf can be written as an infinite union of diffeomorphic pieces and the projection is injective on each piece.

\begin{definition}
\label{def:Spiraling}
A neighborhood of an end $B \subset L$ is said to \emph{spiral} onto a leaf $\wtd{L}$ if there exists a projection $\pi:B \to \wtd{L}$ (obtained by projecting along the leaves of an auxiliary transverse $1-$dimensional foliation), and a (closed, connected) codimension-$1$ submanifold $N \subset L$ called the \emph{juncture} such that:
\begin{enumerate}[label=(\roman*)]
    \item There exists a decomposition $B = \bigcup_{j=0}^\infty B_j$, with $\partial B_j = N_{j} \sqcup N_{j+1}$ and $\operatorname{int}(B_i) \cap \operatorname{int}(B_j) = \emptyset$ if $i \neq j$;
    \item The projection $\pi$ maps each $N_i$ diffeomorphically onto $N$ for each $i \in \N$;
    \item The restriction of $\pi$ to $B_i \setminus N_{i+1}$ is injective for each $i \in \N$;
    \item For each $p \in \wtd{L}$, the sequence $\pi^{-1}{p} = \{q_i\}_{i \in \N}$, where $q_i \in B_i$, converges monotonically (w.r.t. the coorientation of $\F$) to $p$.
\end{enumerate}
\end{definition}
Evidently, the interior of each $B_i$ is diffeomorphic to $L\setminus N$, and $B$ is an \emph{infinite repetition} of $L$:
\[
B = B_0 \cup_{N_1} B_1 \cup_{N_2} B_2 \cup_{N_3} \ldots
\]
In this case we also say that the e (of which $B$ is a neighborhood) is \emph{periodic with period $L$}.

\Cref{thm:SpiralingTheorem} has a nice consequence on the asymptotic behaviour of proper leaves with finitely many ends, which appeared as part of \cite[Theorem 2.5]{MCS20}. 
For the reader's convenience, we restate it here explicitly and give a detailed proof.

\begin{prop}
\label{prop:proper_depth_one}
Let $L$ be a proper leaf with finitely many ends. 
Then, $L$ is totally proper and $\depth(L) = 1$.
\end{prop}

\begin{proof}
We start by proving that $L$ is a \emph{totally proper} leaf, meaning that each leaf in the closure $\overline{L}$ is proper.
Since $L$ is a proper leaf, it is a local minimal set \cite[Proposition 8.1.19]{CanCon00}. 
Hence its closure $\overline{L}$ is a finite union of local minimal sets \cite[Corollary 8.3.12]{CanCon00}.
Recall now that a local minimal set can be of three types:
\begin{enumerate}
\item an open saturated set of $\F$;
\item a single proper leaf (cf.\ \cite[Proposition 8.1.19]{CanCon00} again);
\item an exceptional local minimal set, by which we mean that its closure is transversely a Cantor set.
\end{enumerate}

We claim that all the minimal sets of $\overline{L}$ are of the second type.
First, observe that a local minimal set contained in $\overline{L}$ cannot be of the first type, this would violate the properness of $L$. 
Moreover, Duminy's Theorem  \cite{CanCon02} tells us that $\overline{L}$ does not contain any exceptional minimal sets. 
We conclude that every leaf in $\overline{L}$ is proper, that is, $L$ is a totally proper leaf.

To see that $L$ is at depth $1$, observe that \cite[Corollaries 8.3.10 and 8.3.16]{CanCon00} imply that any totally proper leaf is at finite depth. Since $L$ has finitely many ends, one can then conclude from \cite[Corollary 8.4.7]{CanCon00} that $\depth(L) =1$. 
\end{proof}

For symplectic foliations the above discussion implies that the ends of totally proper leaves are not only smoothly periodic but also symplectically. We consider two notions of periodicity for symplectic leaves.

\begin{definition}
    \label{def:sympl_periodic_end}
An end $e$ of a symplectic manifold $(W,\omega)$ is called:
\begin{itemize}
        \item \emph{symplectically periodic} if it can be represented by a sequence $\{h^n(U)\}_{n \in \N}$ where $h:U \to h(U) \subset U$ is a symplectomorphism, and $U$ a neighborhood of $e$;
        \item \emph{symplectically almost periodic} if it can be represented by a sequence $\{h^n(U)\}_{n \in \N}$ where $h:U \to h(U) \subset U$ is a diffeomorphism and $U$ a neighborhood of $e$ such that: there exists a symplectic form $\omega_\infty \in \Omega^2(\overline{U} \setminus h(U))$ satisfying
        \[ (h^n)^*\omega \xrightarrow{n \to \infty} \omega_\infty,\]
        where the convergence is with respect to the $C^0$-Whitney topology on the compact set $\overline{U}\setminus h(U)$.
    \end{itemize}
\end{definition}

As the nomenclature suggest, a symplectically period end is in particular symplectically almost periodic.
The smoothness of the leafwise symplectic form together with Theorem \ref{thm:SpiralingTheorem} then immediately implies:
\begin{corollary}
\label{cor:finite_depth_sympl_approx_period}
Let $L$ be a leaf of a symplectic foliation $(\F,\omega)$ on a compact manifold. If $L$ is at finite depth then it has symplectically almost periodic ends.
\end{corollary}

Although we will not need it explicitly in the following, we point out that, in the setting of \emph{strong} symplectic foliations, one arrange that the fibers of the projection $\pi$ in Definition \ref{def:Spiraling} are tangent to the kernel of any closed extension of the leafwise symplectic form to the ambient manifold.
As the flow of any vector field in such kernel foliation preserves the leafwise symplectic structure, one gets the following stronger variant  of \Cref{cor:finite_depth_sympl_approx_period} in the strong symplectic foliated setting:
\begin{corollary}
    \label{cor:finite_depth_sympl_period}
    Let $L$ be a leaf of a strong symplectic foliation $(\calF,\omega)$ on a compact $M$.
    If $L$ is totally proper and at depth $1$, then $(L,\omega\vert L)$ has symplectically periodic ends.
\end{corollary}

By combining \Cref{prop:proper_depth_one,cor:finite_depth_sympl_approx_period,cor:finite_depth_sympl_period} above, we conclude:
\begin{thm}
    \label{thm:sympl_proper_leaves_approx_periodic_end}
    Let $(W,\omega)$ be an open symplectic manifold with a finite number of ends.
    If it is symplectomorphic to a proper leaf of a symplectic foliation,
    then $\omega$ is symplectically almost periodic.
    Moreover, if the symplectic foliation is strong, then $\omega$ is symplectically periodic.
\end{thm}

\section{Geometrically bounded symplectic manifolds}
\label{sec:geom_bound}

Although the leaves we consider are non-compact, the compactness of the ambient manifold implies that the leafwise symplectic structure behaves as on compact manifolds.

\begin{definition}
    \label{def:geom_bounded}
    A symplectic form $\omega$ on $W^{2n}$ is said to be \emph{\acf{gb}} if there are an almost
    complex structure $J$ and a complete Riemannian metric $g$ satisfying the following conditions.
    \begin{enumerate}[label=(GB\arabic*)]
	\item \label{GB1} There are strictly positive constants $A,B$ such that, for each $u,v\in TW$,
	    \begin{equation*}
		\omega(u,Ju) \geq A \normg{u}^2 , \quad \quad
		\norm{\omega(u,v)}\leq B \normg{u}\normg{v} .
	    \end{equation*}
	\item \label{GB2} The metric $g$ has sectional curvature $K_g$ bounded from above and injectivity radius $\inj(W,g)$
	    bounded from below by a positive constant.
    \end{enumerate}
    In this case, we also call $(\omega,J,g)$ a \emph{\acf{gb} triple}.
\end{definition}

\begin{remark}
\label{rmk:completeness_useless}
Asking that the Riemannian metric $g$ is complete in above definition is actually redundant. 
Indeed, any Riemannian metric with injectivity radius bounded below by a strictly positive constant is complete; cf.\ \cite[Lemma 2.1]{Lu98}.
\end{remark}

Each almost complex structure $J$ which is compatible with $\omega$ defines a Riemannian metric given by $g_{\omega,J} = \omega(\cdot , J\cdot)$.
However, in general the triple $(\omega,J,g_{\omega,J})$ might not be \acs{gb} if $W$ is open; although condition \ref{GB1} is trivially satisfied (with constants $A=B=1$), condition \ref{GB2} might not be. 

\begin{definition}
    \label{def:str_geom_bounded}
    A symplectic manifold is said to be \emph{\acf{sgb}} if there is a compatible almost complex structure $J$
    for which $(\omega,J,g_{\omega,J})$ is a geometrically bounded triple.
\end{definition}

\begin{lemma}
    \label{lemma:sympl_leaves_sgb} 
    Leaves of symplectic foliations on closed manifolds are \acs{sgb}.
\end{lemma}

\begin{proof}
This can be proven with a polar decomposition argument as in the non-foliated case (c.f.\ for instance \cite[Proposition 12.3]{CanDaSilBook}); we give a sketch of the argument.
Let $\omega\in\Omega^2(\calF)$ be the leafwise symplectic form, and fix an auxiliary Riemannian metric $g$ on the ambient manifold $M$.
By non-degeneracy of both $\omega$ and $g$ on $T\calF$, there is an endomorphism $\psi \colon T\calF \to T\calF$ such that $g(\psi \cdot , \cdot) = \omega$.
One can then explicitly check that $\psi$ is skew-symmetric and that the composition $\psi \circ \psi^*$ of $\psi$ with its $g-$adjoint $\psi^*\colon T\calF\to T\calF$ is symmetric and positive-definite.
In particular, it admits a square root $ \phi\coloneqq \sqrt{\psi\circ \psi^*}$, which is also symmetric and positive-definite.
It follows that the endomorphism $J\coloneqq \phi^{-1}\circ\psi $ of $T\calF$ is a leafwise almost complex structure.
We then define the Riemannian metric $g_{\omega,J}\coloneqq \omega(J\cdot , \cdot)$ on $T\calF$.
The compactness of the ambient manifold then implies $(\omega,J,g_{\omega,J})$ is a \acs{sgb} triple, as desired.
\end{proof}

Apart from (universal covers of) closed symplectic manifolds, standard examples in the literature of \acs{gb} manifolds are (twisted) contangent bundles and Liouville completions of symplectic fillings.
All the non-compact examples have a (convex) cylindrical endover a contact manifold at infinity. As such these manifolds are not symplectically (almost) periodic and cannot be leaves of a symplectic foliation. 

\begin{definition}\label{def:ContactExhaustion}
An \emph{exhaustion of contact type} on a symplectic manifold $(W,\omega)$ is a collection of compact sets $\calK = \{K_n\}_{n \in \NN}$ such that:
\begin{enumerate}[label=(\roman*)]
    \item $K_0 =\emptyset$, $W = \cup_{n \in \NN} K_n$ and $K_{n} \subset \sint K_{n+1}$;
    \item $(K_n,\omega|_{K_n})$ is a symplectic domain with (smooth) boundary of convex contact type.
\end{enumerate}
\end{definition}
The main class of examples admitting a exhaustion of contact type is given by Liouville manifolds of infinite type.

The rest of this section deals with the proof of Theorem \ref{thm:strongly_geom_bounded} which states that admitting a exhaustion of contact type is a sufficient condition for being \acs{sgb}.
The proof uses two ingredients, given in the following two lemmas. On the compact sets $K_{n+1} \setminus K_n$ we can rescale $g$ to satisfy any bound on the curvature and injectivity radius . The resulting metrics are then glued by interpolating. To ensure that the interpolation preserves the \acs{sgb} condition we show that we can insert an arbitrarily large part of the symplectization of $\partial K_{n}$. This is where Definition \ref{def:ContactExhaustion} is used.

\begin{lemma}\label{lem:RiemannianScaling}
Let $(M,g)$ be a Riemannian manifold and $k>0$ a constant.
Then:
\[
K_{kg} = k^{-1} K_{g},
\quad
\inj(M,kg) = k\,\inj(M,g).
\]
\end{lemma}
In particular, this lemma shows that if \Cref{def:str_geom_bounded} is satisfied for some $g$ then it is satisfied for all $e^cg$ with $c\geq 0$.

Recall that given a contact manifold $(M,\xi = \ker\alpha)$ we can form the symplectization
\[ \left( M \times \R, \omega = \d(e^t \alpha)\right).\]
Since $(\xi,\d \alpha|_{\xi})$ is a symplectic vector bundle we can find a compatible almost complex structure $J_\xi$, which in turn gives a metric $g_\xi = \d \alpha\vert_\xi (\cdot , J_\xi \cdot)$. 
Then,
\begin{equation}
\label{eq:SymplectizationCompatibleTriple} 
\omega = \d(e^t\alpha),\quad g := e^t(g_\xi + \alpha\otimes \alpha + \d t\otimes \d t),\quad J := J_\xi + R\otimes \d t - \partial_t \otimes \alpha,
\end{equation}
defines a compatible triple on $M \times \R$.

\begin{lemma}
\label{lem:SymplecticScaling}
Let $(M,\xi = \ker \alpha)$ be a contact manifold, and $(\omega,g,J)$ be the associated compatible triple on $M\times (-\epsilon,\epsilon)$ from Equation \ref{eq:SymplectizationCompatibleTriple}.
Then, for any constants $a<b$, there exists another compatible triple $(\wtd{\omega},\wtd{g},\wtd{J})$ on $M\times(-\epsilon,\epsilon)$ satisfying:
\begin{enumerate}[label=(\roman*)]
    \item \label{cond:ScalingI} On $M\times(-\epsilon,-\epsilon/3)$, $(\wtd{\omega},\wtd{g},\wtd{J}) = (e^a\omega,e^ag,J)$.
    \item \label{cond:ScalingII} On $M\times(\epsilon/3,\epsilon)$. $(\wtd{\omega},\wtd{g},\wtd{J}) = (e^b\omega,e^bg,J)$
    \item \label{cond:ScalingIII} The metric $\wtd{g}$ satisfies:
    \[K_{\wtd{g}}
    \leq e^{-a+\epsilon/3}
    K_g
    ,\quad\inj(\wtd{g},M\times (-2\epsilon/3,2\epsilon/3)) \geq e^a \min(\epsilon/2,\inj(g\vert_{M\times\{0\}})) ,
    \]
    where $\inj(\wtd{g},M \times (-2\varepsilon/3,2\varepsilon/3)$ means the injectivity radius of $\wtd{g}$ considering geodesic balls centered at points in $M \times (-2\varepsilon/3,2\varepsilon/3)$ and contained in $M \times(-\varepsilon,\varepsilon)$.
\end{enumerate}
\end{lemma}
Notice that there is an overlap between the regions where $\wtd{g}$ is a multiple of $g$ and where there is a positive lower bound on the injectivity radius for $\wtd{g}$.
This will be important in the proof of \Cref{thm:strongly_geom_bounded} below.

\begin{proof}[Proof of \Cref{lem:SymplecticScaling}]
Let $\psi:(-\varepsilon,\varepsilon) \to (a-\varepsilon,b+\epsilon)$
be a smooth diffeomorphism such that
$\psi(t)=t+a$ for $t\in(-\varepsilon,-\varepsilon/3)$, and $\psi(t)=t+b$ for $t\in(\varepsilon/3,\varepsilon)$.
We then claim that $(\wtd{\omega},\wtd{g},\wtd{J}) := (\psi^*\omega,\psi^*g,\psi^*J)$ on $M\times (-\varepsilon,\varepsilon)$ satisfies the desired properties. 

Using the explicit form of $\psi$ near the boundary, Conditions \ref{cond:ScalingI} and \ref{cond:ScalingII} are easily checked.
For Condition \ref{cond:ScalingIII} we use that, for $0<c<b-a$, the translation
\[\tau_c\colon (a-\varepsilon,b-\varepsilon-c) \to (a-\varepsilon,b-\varepsilon), \quad t \mapsto t + c,\]
satisfies $\tau_c^*g = e^cg$.  
Together with Lemma \ref{lem:RiemannianScaling} this shows that
\[ K_{g}\vert_{M\times (a-\epsilon,b+\epsilon)} \leq e^{-a+\varepsilon/3}K_{g}\vert_{M \times (a-\varepsilon,a-\varepsilon/3)}. 
\]
Then, as the restriction of $g$ to $M \times (a-\varepsilon,a - \varepsilon/3)$ pulls back (via $\psi)$ to $e^ag$, the curvature bound follows. 

It remains to show the injectivity radius is bounded.
For each point $p \in M \times (-\varepsilon,\varepsilon)$, let $\rho(p)$ denote the maximal radius of a $g$-geodesic ball centered at $p$ and entirely contained in $M \times (a-\varepsilon,b+\varepsilon)$. 
On $M \times (a-\varepsilon,a - \varepsilon/3)$ we have $\wtd{g} = e^ag$, so that
\[ \min_{p \in M \times \{a-2\varepsilon/3\}} \rho(p) \geq e^a\min(2\varepsilon/3,\inj(g|_{M \times \{0\}}));\]
here, we used that $\inj(g|_{M \times \{a-2\varepsilon/3\}})=\inj(g|_{M \times \{0\}})$, as $g\vert_{M\times\{t\}}$ is independent of $t$ (under the natural identification $M\times\{t\}=M$).

Now, according to \Cref{lem:RiemannianScaling} and the fact that $\tau_c^*g=e^cg$ as pointed out above, $\rho$ is at least $\min_{p \in M \times \{a-2\varepsilon/3\}} \rho(p)$ on all of $M\times[a-2\varepsilon/3,b+2\varepsilon/3]$;
this proves the desired bound on the injectivity radius, thus concluding the proof.
\end{proof}

\begin{proof}[Proof of \Cref{thm:strongly_geom_bounded}]
By assumption the boundaries $\partial K_n$, $n\in \N$, are hypersurfaces of contact type. Hence we can fix collar neighborhoods
\begin{equation}
\label{eqn:neighborhoods_contact_hypersurf}
    \left(\partial K_n \times (-\varepsilon_n,\varepsilon_n), \omega = \d (e^t \alpha_n) \right).
\end{equation}
On these neighborhoods, Equation \ref{eq:SymplectizationCompatibleTriple} defines a compatible triple $(\omega_n=\d(e^t\alpha_n),g_n,J_n)$, and we extend this to a compatible triple $(\omega,g,J)$ on $M$.

Using Lemma \ref{lem:RiemannianScaling} we choose a strictly increasing sequence of constants $\{k_n\}_{n \in \N}$ such that
\begin{equation}\label{eq:GreeneExhaustionBounds} K_{e^{k_n}g} < 1,\quad \inj(e^{k_n}g) > 1 \quad \text{ on }  K_n \setminus \sint (K_{n-1})
\end{equation}
and
\[
e^{k_n} \min(\epsilon_n/2,\inj(g\vert_{\partial K_n})) >1 .
\] 
For each of the neighborhoods in Equation \ref{eqn:neighborhoods_contact_hypersurf} we apply Lemma \ref{lem:SymplecticScaling} with $a = k_n$ and $b = k_{n+1}$. 

The resulting compatible triple $(\omega',g',J')$ has bounded injectivity radius since the regions where the injectivity radius are bounded by Lemma \ref{lem:SymplecticScaling} and \ref{eqn:neighborhoods_contact_hypersurf} respectively have non-trivial overlap. The other conditions of Definition \ref{def:str_geom_bounded} are easily verified.
\end{proof}

\section{Manifolds not diffeomorphic to symplectic leaves}
\label{sec:not_diffeo_sympl_leaves}

In this section we prove \Cref{thm:smooth_manifold_sympl_nonleaf} and \Cref{cor:R2n_non_proper_leaf}. 
The core of the proof consists of a volume argument inspired by \cite{Ber01}. However, the implementation is slightly different for proper and non-proper leaves. Therefore, we separate the argument into two propositions:

\begin{prop}
    \label{prop:smooth_manifold_sympl_non_proper_leaf}
Let $W^{2n}$ be an open manifold with a finite number of ends. Suppose at least one of the ends has a neighborhood of the of the form $N 
\times [0,\infty)$ where $N$ has trivial $\pi_1$ and $H^2$. Then $W$ is not diffeomorphic to a proper leaf of a symplectic foliation.
\end{prop}

\begin{prop}
 \label{prop:smooth_manifold_sympl_non_nonproper_leaf}
    Let $W^{2n}$ be an open manifold with a finite number $k$ of ends. 
    Suppose that, for $i=1,\dots,k$, the $i-$th end has a neighborhood of the form $N_i \times [0,\infty)$, where $N_i$ has trivial $\pi_1$ and $H^2$. 
    If there is only one end, \emph{additionally} assume that $W \setminus N_1 \times (0,\infty)$ has non-trivial $\pi_1$ or non-trivial $H^2$. 
    Then $W$ is not diffeomorphic to a non-proper leaf of a symplectic foliation.
\end{prop}

\begin{proof}[Proof of \Cref{thm:smooth_manifold_sympl_nonleaf}]
It is enough to combine \Cref{prop:smooth_manifold_sympl_non_proper_leaf,prop:smooth_manifold_sympl_non_nonproper_leaf} above.
\end{proof}

\begin{proof}[Proof of \Cref{prop:smooth_manifold_sympl_non_proper_leaf}]
Suppose by contradiction that $W$ is diffeomorphic to a proper leaf $L$ of a symplectic foliation $(M,\F,\omega)$.
Let $e$ be the end of $L$ having a neighborhood of the form $\mathcal{N}:= N \times [0,\infty)$, as in the assumptions of \Cref{prop:smooth_manifold_sympl_non_nonproper_leaf}. 
According to \Cref{prop:proper_depth_one} and \Cref{thm:SpiralingTheorem}, $e$ spirals onto a closed leaf $L_\infty$.
Let us assume, for the moment, that the associated projection $\pi:\mathcal{N} \to L_\infty$ (as in Definition \ref{def:Spiraling}) maps $N \times \{0\}$ to an \emph{embedded} submanifold $N_\infty \subset L_\infty$;
we will explain how to deal with the general case at the end of the proof.

Since $H^1(N) = 0$, Reeb stability allows us to find a foliated subset $ U:= N \times [0,1]^2$ such that
\[ \F|_U = \bigcup_{z \in [0,1]} N \times [0,1] \times \{z \},\quad L_\infty \cap U = N \times [0,1] \times \{0\}, \quad \mathcal{N} \cap U = \bigcup_{k \in \N} N \times [0,1] \times \{1/k\}.\]
Moreover, there exist differential forms $\alpha \in \Omega^1(\calF\vert_U)$ and $\beta \in \Omega^1(\mathcal{N})$ satisfying:
\[ \omega|_U = \d \alpha,\quad \omega|_{\mathcal{N}} = \d \beta,\quad \alpha|_{\mathcal{N} \cap U} = \beta\vert_{\mathcal{N} \cap U}.\]
Indeed, $H^2(N) = 0$ implies $H^2(\F|_{U}) = 0$ so that we find a leafwise primitive $\alpha \in \Omega^1(\F|_U)$ for $\omega|_U$. Similarly, $H^2(\mathcal{N}) = 0$ and we find $\wtd{\beta} \in \Omega^1(\mathcal{N})$ such that $\omega|_\mathcal{N} = \d \wtd{\beta}$. On the intersection $U \cap \mathcal{N}$ the difference of the two primitives is a closed form which is exact since $H^1(U \cap \mathcal{N}) = 0$. That is, there exists a function $\wtd{f} \in C^\infty(U \cap \mathcal{N})$ so that
\[(\alpha - \wtd{\beta})|_{U \cap \mathcal{N}} = \d \wtd{f}.\]
Since $U \cap L$ is closed in $L$, we can extend $\wtd{f}$ to a function $f$ on $L$. Taking $\beta := \wtd{\beta} + \d f$ proves the claim.

Choose now the sequence of submanifolds $N_m = N\times\{1/2\}\times\{1/m\} \subset U$, for $m>0$; notice that, by choice of $U$, $N_m\subset \calN$ as well. 
Notice that, as we are dealing with the case where $N_\infty$ is embedded in $L_\infty$, $N_\infty$ can be seen as the juncture of the spiraling of the end $e$ around $L_\infty$.
In particular, up to modifying the projection map $\pi$ of the spiraling of $e$ onto $L_\infty$ in such a way that it coincides with the projection onto the first two factors $N\times[0,1]$ on $U=N\times[0,1]^2$ (notice that this can be achieved up to shrinking $U$ and up to passing to a subsequence of the $N_m$'s; c.f.\ \Cref{def:Spiraling}), we can assume that all the lifts $N_m$ project down via $\pi$ onto the juncture $N_\infty$.
In particular, all the $N_m$'s are separating in the end $e$.
Denote then by $K_m$ the compact subset of $\calN=N\times[0,\infty)$ bounded $N_m$ and $N_1$.
Notice that $N_m$ also escapes at  infinity along the end $e$, so that the closure $P$ of the unbounded subset of $N\times[0,\infty)\setminus N_1$ is just given by the ``limit'' of $K_m$ for $m\to\infty$, i.e.\ by $\cup_{m\geq1} K_m$.

As $U$ is compact and $\omega$ is continuous, the $\omega$-volume of each of the plaques $N \times [0,1] \times \{z\}$ of $U$ is bounded from below by a strictly positive constant. 
Since $P$ intersects $U$ in infinitely many of these plaques and $K_m\subset K_{m+1}$, we have
\[ 
\lim_{m\to\infty} \int_{K_m} \omega^n = 
\int_{P}\omega^n = 
\infty.
\]
Using Stokes' theorem and the fact that $\alpha=\beta$ on $\calN\cap U$, we then get
\[ 
\lim_{m \to \infty} \int_{N_m} \alpha \wedge \omega^{n-1} = 
\lim_{m \to \infty} \int_{N_m} \beta \wedge \d \beta^{n-1} = 
\infty.
\]
On the other hand, because of the continuity of $\alpha,\omega$ on $U$ and the choice of $N_m=N\times\{1/2\}\times\{1/m\}\subset U$, 
we must have
\[ 
\lim_{m \to \infty} \int_{N_m} \alpha \wedge \omega^{n-1} = 
\lim_{m \to \infty} \int_{N \times \{1/2\}\times \{1/m\}} \alpha \wedge \omega^{n-1} = \int_{N_\infty} \alpha \wedge \omega^{n-1} \in \R.
\]
We have thus reached a contradiction, proving that $W$ cannot be diffeomorphic to a proper symplectic leaf as desired.

It remains to explain how to deal with the case where the projection of $N \times \{0\}$ under $\pi$ is \emph{not} embedded in $L_\infty$. In general $N \times \{0\}$ is contained in the union $B_i \cup \dots \cup B_{i + k}$ for some $i \in \N$ and $k > 1$ (the projection is embedded precisely when $k=1$), see \Cref{def:Spiraling} for notation. 
Let $p:L_{\infty,k} \to L_{\infty}$ be the $k$-cover obtained by cutting $L_\infty$ along a juncture, gluing $k$ copies of this cut manifold one after the other, and gluing the two remaining free boundary components to create a closed manifold $L_{\infty,k}$.
Next, extend $p$ to an immersion \[\wtd{p}:L_{\infty,k} \times (-\varepsilon,\varepsilon) \to \Op(L_\infty),\] for some $\varepsilon > 0$, and consider the pullback (symplectic) foliation. 

Each connected component of the preimage of the neighborhood $\mathcal{N}$ (which we may assume to be contained in the image of $\wtd{p}$ up to taking a smaller neighborhood of $e$ of the same form $N\times[0,\infty)$) is diffeomorphic to $N \times [0,\infty)$, and spirals onto $L_{\infty,k}$.
However, since the induced projection is injective on $B_{i} \cup \dots \cup B_{i+k}$, now $N \times \{0\}$ does project to an embedded submanifold of $L_{\infty,k}$. 
Then, the argument previously described shows that $N \times [0,\infty)$ cannot be (part of) a leaf in the pullback foliation on $L_{\infty,k} \times (-\varepsilon,\varepsilon)$, thus giving the desired contradiction.
\end{proof}

\begin{proof}[Proof of \Cref{prop:smooth_manifold_sympl_non_nonproper_leaf}]
As $L$ is non-proper, there is an end $e$ of $L$ such that any distinguished neighborhood $\mathcal{N} := N\times [0,\infty)$ of $e$ accumulates onto itself. 
Indeed, the limit set of a leaf is the union of the limit sets of its ends, and each of the latter is a saturated set; c.f.\ for instance \cite[Lemmas 4.3.5 and 4.3.7]{CanCon00}.

Then, using Reeb stability and $\pi_1(N) =0$, we find a foliated subset $U := N \times [0,1]^2$ of $M$ such that
\begin{equation}\label{eq:NonProperFoliatedChart}
 \F|_U = \bigcup_{z \in [0,1]} N \times [0,1] \times \{z\},\quad \mathcal{N} \cap U  \supset N \times [0,1] \times \{0\} \cup \bigcup_{m \in \N} N \times [0,1] \times \{s_m\},
 \end{equation}
for a strictly decreasing sequence $\{s_m\}_{m \in \N}$.
The submanifolds $N \times \{1/2\} \times \{s_m\}$ give a sequence of submanifolds $N_m\subset \mathcal{N}$
converging (in the ambient manifold) to $N_\infty\coloneqq N\times \{1/2\}\times \{0\} \subset U$.

Now, the same argument as in the proof of \Cref{prop:smooth_manifold_sympl_non_proper_leaf} shows that there exist differential forms $\alpha \in \Omega^1(\mathcal{F}\vert_U)$ and $\beta \in \Omega^1(\mathcal{N})$ satisfying:
\[ \omega|_{U} = \d \alpha,\quad \omega|_{\mathcal{N}} = \d \beta,\quad \alpha|_{\mathcal{N} \cap U} = \beta.\]

As $\alpha$ and $\omega$ are continuous over $N_\infty$, we have that
\begin{equation}
\label{eqn:finite_integral}
\lim_{m \to \infty} \int_{N_m} \alpha \wedge \omega^{n-1}
=\int_{N_\infty} \alpha \wedge \omega^{n-1}
\in \R.
\end{equation}

It turns out that all the $N_m$ represent the same (non-zero) cohomology class in $\mathcal{N}$. We now conclude the proof under this assumption; afterwards we will prove the claim.

The assumption implies that there exist compact subsets $K_m \subset \mathcal{N}$ such that 
\[ \partial K_m = N_m - N_0.\]
The plaques in Equation \ref{eq:NonProperFoliatedChart} have $\omega$-volume bounded from below by a strictly positive constant. Therefore, the $\omega$-volume of the $K_m$ goes to infinity. 
Then, Stokes' theorem implies
\[ \lim_{m \to \infty} \int_{N_m} \beta \wedge \omega^{n-1} - \int_{N_0} \beta\wedge \omega^{n-1} = \lim_{m \to \infty} \int_{K_m} \omega^{n}  = \infty,\]
which contradicts \Cref{eqn:finite_integral} as $\alpha=\beta$ over $U\cap \calN$, thus concluding the proof.

It remains to show that the previously made assumption is always met, i.e.\ that, up to passing to a subsequence, the $N_m$ all represent the same (non-trivial) homology class as $N_\infty = N\times\{1/2\}\times\{0\}$ in $\mathcal{N}$.
As each $N_m$ is connected and $H^{2n-1}(\mathcal{N}) = \ZZ$ generated by $N_\infty$, it follows (see for instance \cite{MeePat77}) that $[N_m]\in H^{2n-1}(\mathcal{N};\ZZ)$ is equal to either $[N_\infty]$ or $0$.

So, assume by contradiction that $[N_{m_0}]=0$ in $H^{2n-1}(\mathcal{N};\ZZ)$ for some $m_0\in \N$;
denote by $C_{m_0}$ the compact region (contained inside $\mathcal{N}$) which is bounded by $N_{m_0}$. We may assume that $N_{m_0}$ bounds $C_{m_0}$ ``from the left'', by which we mean
\[ C_{m_0} \cap U = N \times [1/2,1] \times \{s_{m_0}\},\]
i.e. $C_n$ sits on the right of $N_n$ inside $U$.

We consider then the set 
\[ \mathcal{S} := \left\{ s \in [0,1] \Big\vert \,\, \parbox{8cm}{$ N \times \{1/2\} \times \{s\} \subset U$ bounds from the left a leafwise submanifold diffeomorphic to $C_{m_0}$,}\right\}.\]
Notice that $\mathcal{S}$ is open and does not contain $0$. 
Indeed, since $N_{m_0}$ and $\mathcal{N}$ are simply connected and have trivial second cohomology the same holds for $C_{m_0}$.
Thus, if $s\in\mathcal{S}$, by Reeb stability there exists an open (foliated) product neighbourhood around the $C_s$ bounded by $N_s=N\times\{1/2\}\times\{s\}$ and diffeomorphic to $C_{m_0}$; hence, $C_s$ can be pushed-off to nearby leaves, proving that $\mathcal{S}$ is open. 
Furthermore, by the assumptions of \Cref{prop:smooth_manifold_sympl_non_nonproper_leaf}, $N_\infty$ does not bound any compact, simply-connected region without second cohomology, so that $0 \notin \mathcal{S}$.

Let $s_*$ be the infimum of the connected component of $\mathcal{S}$ containing $s_{m_0}$ and notice that, $s_*\notin \mathcal{S}$ and $s_*> 0$.

We now use the following theorem due to Schweitzer \cite{Sch11}:

\begin{thm}[{\cite[Proposition 7.1]{Sch11}}]
Let $(M,\F)$ be a foliated manifold and $C$ be a compact manifold with boundary $B$. If $(C\times (0,1] \cup B \times \{0\}) \to (M,\F)$ is a foliated embedding that cannot be extended over $C \times [0,1]$, then the leaf containing $h(B \times \{0\})$ is the boundary of a generalized Reeb component whose interior is the union of the leaves meeting $h(B \times (0,1])$.
\end{thm}
The fact that a foliated embedding as in the hypothesis exists in our setting simply follows from a continuation argument using Reeb stability starting from $C_{m_0}$.
Then, the result readily implies that $L$ is either the boundary of a generalized Reeb component, in the case where $N_{s_*}$ is in $L$, or in the interior of a generalized Reeb component, in the case where  $N_{s_*}$ is in a leaf $L_*$ different from $L$.
In both cases, $L$ must be proper, thus giving the desired contradiction.
\end{proof}

We now turn our attention to \Cref{cor:R2n_non_proper_leaf} concerning $\RR^{2n}$ and its realizability as a symplectic leaf.
\begin{proof}[Proof of \Cref{cor:R2n_non_proper_leaf}]
The fact that $\RR^{2n}$ is not diffeomorphic to a proper symplectic leaf simply follows from \Cref{prop:smooth_manifold_sympl_non_proper_leaf} above.
It's then enough to realize $(\RR^{2n},\omega_{std})$ as a dense leaf in some symplectic foliation; this can be done as follows.

Start from a foliation of $\RR^{2n+1}$ by affine hyperplanes $\RR^{2n}$, pairwise parallel to each other and each directed by a $2n$-dimensional vector subspace of irrational and $\QQ$-linearly independent slopes. 
In other words, the directing vector subspaces are defined by $z = \sum_i(a_i x_i + b_i y_i)$, where we use coordinates $(x_i,y_i,z)\in\RR^{2n+1}$ and the $(a_i,b_i)\in\RR^{2n}$ are linearly independent over $\QQ$.
Now, one can pull back, to each of these affine subspaces, the standard symplectic form on $\RR^{2n}$ by the natural projection $\RR^{2n+1}\to\RR^{2n}$; this gives a leafwise symplectic form on this irrational foliation on $\RR^{2n+1}$.
Now, notice that both the foliation and the leafwise symplectic structure are invariant under the natural action of $\ZZ^{2n+1}$ by translation on each coordinate $(x_i,y_i,z)$ of $\RR^{2n+1}$. 
What's more, because of the $\QQ$-linear independence condition on the coefficients $a_i$ and $b_i$, the resulting leaves are just symplectomorphic to $(\RR^{2n},\omega_{std})$, and each of them is dense in $\TT^{2n+1}$, as desired.
\end{proof}

\section{Examples of symplectic non-leaves via blowups}
\label{sec:nonleaves_blowup}

The aim of the section is to prove \Cref{thm:blowup_non_leaves} and \Cref{cor:real_space_blown_up}.

\begin{proof}[Proof of \Cref{thm:blowup_non_leaves}]
Fix an exhaustion of contact type $\{K_m\}_{m \in \N}$ of $(W,\omega)$ (Definition \ref{def:ContactExhaustion}). 
Using \Cref{thm:strongly_geom_bounded}
we find a symplectic form $\omega'$ on $W$ which away from the $K_m$ equals a very large rescaling of $\omega$; 
these rescaling factors are in particular taken so big that there exists a sequence balls
\[ B_m \subset K_m \setminus K_{m-1},\]
such that $(B_m,\wtd{\omega}|_{B_m})$ is a standard symplectic ball of radius $m$. 

Performing a blowup of weight $m$ at $B_m$ for each of the balls (as defined in \cite[Theorem 7.1.21]{McDSalBook}) and subsequently applying Theorem \ref{thm:strongly_geom_bounded} then yields a \acs{sgb} symplectic manifold $(W',\omega')$. 
Notice that (by the explicit formula in \cite[Theorem 7.1.21]{McDSalBook}), the complex projective $(n-1)-$space $C_m$ resulting from the blow-up at the ball $B_m$ with weight $m$ has $\omega'-$volume $m^4\pi^{2n}$.
We claim $(W',\omega')$ is not symplectomorphic to a leaf of a symplectic foliation.

Firstly assume by contradiction it is symplectomorphic to a \emph{proper} symplectic leaf. 
Then, according to \Cref{prop:proper_depth_one} and \Cref{cor:finite_depth_sympl_approx_period} the end of $W'$ is symplectically almost periodic. 
As such, there exists a fundamental system of neighbourhoods $\{h^k(U)\}_{k \in \N}$ as in Definition \ref{def:sympl_periodic_end}. 
In particular, if $C$ is one of the blown-up complex projective spaces entirely contained in $U$, then for any $\varepsilon>0$ there is an $N>0$ such that the $\omega'$-symplectic volume of all the $h^n(C)$, for $n>N$, differ by at most $\varepsilon$.
However, because $H_{2n-2}$ of $W'$ is generated by the classes of the complex projective spaces coming from the blow-ups, their $\omega'-$volumes must differ by at least $1$ by choice of the blow-up weights, so we arrive at a contradiction.

Secondly, assume by contradiction that $(W',\omega')$ is symplectomorphic to a \emph{non-proper} leaf $L$ of a symplectic foliation. 
Fix a simply-connected neighborhood $\mathcal{C}\subset L$ of one of the $C_m$'s. 
By Reeb stability we find a foliated subset $\mathcal{C} \times (-1,1)$. 
As $L$ is non-proper, it intersects $\mathcal{C} \times (-1,1)$ in infinitely many plaques $\mathcal{C} \times \{pt\}$. 
By continuity of the leafwise symplectic form, this implies that, for any $\epsilon>0$, $L$ contains infinitely many copies of $\CP^{n-1}$ whose $\omega'$-volume differ by at most $\varepsilon$. 
Thus, as in the proper case, we arrive at a contradiction.
\end{proof}

\begin{proof}[Proof of \Cref{cor:real_space_blown_up}]
Let $W$ be the smooth manifold obtained by complex blow up of $\CC^n$ at infinitely many points, as in the statement. 
The fact that $W$ admits a symplectic form $\omega$ for which $(W,\omega)$ is not symplectomorphic to a symplectic leaf follows from \Cref{thm:blowup_non_leaves}.
We now want to realize as a symplectic leaf a symplectic blow-up $(W,\omega')$ of $\omega_{std}$ on $\CC^n$ for which the resulting complex spaces $C_j$'s are all of the same $\omega'$-volume.

For this, we start from the example of symplectic foliation $(\calF,\omega)$ on $\TT^{2n+1}$ by dense leaves, all symplectomorphic to $(\RR^{2n},\omega_{std})$, as in the proof of \Cref{cor:R2n_non_proper_leaf} in \Cref{sec:not_diffeo_sympl_leaves}.
Recall that such foliation is transverse to the second $\SSS^1$ factor of $\TT^{2n+1}=\TT^{2n}\times\SSS^1$, and that the projection onto the first factor is a local symplectomorphism onto $(\TT^{2n},\omega_{T})$, where $\omega_{T}$ is the symplectic structure on $\TT^{2n}=\RR^{2n}/\ZZ^{2n}$ induced from $(\RR^{2n},\omega_{std})$.

Now, consider any transverse curve $\gamma$, for instance $\gamma(z)=(x_i^0,y_i^0,z)$ for any choice of $(x_i^0,y_i^0)\in\TT^{2n}$.
As the restriction of the projection onto the first factor $\TT^{2n+1}=\TT^{2n}\times\SSS^1\to\TT^{2n}$ to each leaf is a local symplectomorphism, $\gamma$ admits, for $\delta>0$ sufficiently small, a symplectically foliated neighborhood of the form $(B^{2n}_\delta\times\SSS^1,\omega_{std}^{B})$, where $B^{2n}_\delta$ is the ball of radius $\delta$ in $\RR^{2n}$, $\calF$ is of the form $B_\delta \times \{\theta\}$ with $\theta\in\SSS^1$, and the leafwise symplectic form $\omega_{std}^{B}$ is just given by the restriction of $\omega_{std}$ on $\RR^{2n}$ to $B_\delta^{2n}$.
This allows to perform an $\SSS^1$-equivariant symplectic blow-up construction in this local model $(B^{2n}_\delta\times\SSS^1,\omega_{std}^{B})$ (as in \cite[Theorem 7.1.21]{McDSalBook}), in such a way that the origin of the $B^{2n}_\delta$ factor is replaced by a complex projective $(n-1)-$space of a certain symplectic volume $\epsilon$ (the same for every $\theta\in\SSS^1$). 
We denote the result of this blowup by $(X\times\SSS^1,\omega_{X})$.
This glues well to $(\TT^{2n+1},\calF,\omega) \setminus (B^{2n}_\delta,\omega_{std}^{B})$ in order to give a symplectically foliated manifold $(M,\G,\Omega)$.
Notice that, as  the leaves $\RR^{2n}$ of $(\TT^{2n+1},\calF)$ are all dense and $\gamma$ intersects each of them in infinitely many points, the leaves of $\G$ are also dense and are just smoothly obtained by (complex-)blowing up $\RR^{2n}$ at infinitely many points.
Lastly, the restriction of the symplectic form to each leaf is a symplectic form $\omega$ on $W$ such that all the $C_j$'s obtained from blowing up each point have same $\omega$-volume $\epsilon$.
This concludes the proof.
\end{proof}

\bibliographystyle{halpha}
\bibliography{my_bibliography}

\begin{thebibliography}{MnCS18}

\bibitem[AH96]{AttHur96}
Oliver Attie and Steven Hurder.
\newblock Manifolds which cannot be leaves of foliations.
\newblock {\em Topology}, 35(2):335--353, 1996.

\bibitem[AN]{AlbNieInPrep}
Davide Alboresi and Klaus Niederkr\"uger.
\newblock The topology of leaves of a symplectic foliation via
  ${J}-$holomorphic curves.
\newblock In preparation.

\bibitem[Ber01]{Ber01}
M\'{e}lanie Bertelson.
\newblock Foliations associated to regular {P}oisson structures.
\newblock {\em Commun. Contemp. Math.}, 3(3):441--456, 2001.

\bibitem[CC87]{CanCon87}
John Cantwell and Lawrence Conlon.
\newblock Every surface is a leaf.
\newblock {\em Topology}, 26(3):265--285, 1987.

\bibitem[CC00]{CanCon00}
Alberto Candel and Lawrence Conlon.
\newblock {\em Foliations. {I}}, volume~23 of {\em Graduate Studies in
  Mathematics}.
\newblock American Mathematical Society, Providence, RI, 2000.

\bibitem[CC02]{CanCon02}
John Cantwell and Lawrence Conlon.
\newblock Endsets of exceptional leaves; a theorem of {G}. {D}uminy.
\newblock In {\em Foliations: geometry and dynamics ({W}arsaw, 2000)}, pages
  225--261. World Sci. Publ., River Edge, NJ, 2002.

\bibitem[CdS01]{CanDaSilBook}
Ana Cannas~da Silva.
\newblock {\em Lectures on symplectic geometry}, volume 1764 of {\em Lecture
  Notes in Mathematics}.
\newblock Springer-Verlag, Berlin, 2001.

\bibitem[CGK04]{CGK04}
Kai Cieliebak, Viktor~L. Ginzburg, and Ely Kerman.
\newblock Symplectic homology and periodic orbits near symplectic submanifolds.
\newblock {\em Comment. Math. Helv.}, 79(3):554--581, 2004.

\bibitem[Ghy85]{Ghy85}
\'{E}tienne Ghys.
\newblock Une vari\'{e}t\'{e} qui n'est pas une feuille.
\newblock {\em Topology}, 24(1):67--73, 1985.

\bibitem[Gre79]{Gre78}
R.~E. Greene.
\newblock Complete metrics of bounded curvature on noncompact manifolds.
\newblock {\em Arch. Math. (Basel)}, 31(1):89--95, 1978/79.

\bibitem[INTT85]{INTT85}
Takashi Inaba, Toshiyuki Nishimori, Masashi Takamura, and Nobuo Tsuchiya.
\newblock Open manifolds which are nonrealizable as leaves.
\newblock {\em Kodai Math. J.}, 8(1):112--119, 1985.

\bibitem[Lu98]{Lu98}
Guangcun Lu.
\newblock The {W}einstein conjecture on some symplectic manifolds containing
  the holomorphic spheres.
\newblock {\em Kyushu J. Math.}, 52(2):331--351, 1998.

\bibitem[MnCS18]{MCS18b}
Carlos Meni\~{n}o Cot\'{o}n and Paul Schweitzer.
\newblock Exotic open 4-manifolds which are nonleaves.
\newblock {\em Geom. Topol.}, 22(5):2791--2816, 2018.

\bibitem[MP77]{MeePat77}
William~H. Meeks, III and Julie Patrusky.
\newblock Representing codimension-one homology classes by embedded
  submanifolds.
\newblock {\em Pacific J. Math.}, 68(1):175--176, 1977.

\bibitem[MS17]{McDSalBook}
Dusa McDuff and Dietmar Salamon.
\newblock {\em Introduction to symplectic topology}.
\newblock Oxford Graduate Texts in Mathematics. Oxford University Press,
  Oxford, third edition, 2017.

\bibitem[MS18]{MCS20}
Carlos {Meni{\~n}o Cot{\'o}n} and Paul~A. {Schweitzer}.
\newblock {Exotic non-leaves with infinitely many ends}.
\newblock {\em arXiv e-prints}, page arXiv:1808.08864, August 2018, 1808.08864.

\bibitem[Sch11]{Sch11}
Paul~A. Schweitzer.
\newblock Riemannian manifolds not quasi-isometric to leaves in codimension one
  foliations.
\newblock {\em Ann. Inst. Fourier (Grenoble)}, 61(4):1599--1631 (2012), 2011.

\bibitem[Sik91]{Sik91}
Jean-Claude Sikorav.
\newblock Quelques propri\'et\'es des plongements lagrangiens.
\newblock {\em M\'em. Soc. Math. France (N.S.)}, (46):151--167, 1991.
\newblock Analyse globale et physique math{\'e}matique (Lyon, 1989).

\bibitem[Son75]{Son75}
Jonathan~D. Sondow.
\newblock When is a manifold a leaf of some foliation?
\newblock {\em Bull. Amer. Math. Soc.}, 81:622--624, 1975.

\bibitem[SS17]{SchSou17}
Paul~A. Schweitzer and F\'{a}bio~S. Souza.
\newblock Non-leaves of foliated spaces with transversal structure.
\newblock {\em Differential Geom. Appl.}, 51:109--111, 2017.

\bibitem[Zeg94]{Zeg94}
Abdelghani Zeghib.
\newblock An example of a {$2$}-dimensional no leaf.
\newblock In {\em Geometric study of foliations ({T}okyo, 1993)}, pages
  475--477. World Sci. Publ., River Edge, NJ, 1994.

\end{thebibliography}

\end{document}